\documentclass{amsart}
\usepackage{amsmath,amssymb,amsthm, epsfig}

\usepackage{latexsym,amssymb}
\usepackage{amsmath,amsthm}

\usepackage{amsfonts}
\usepackage{graphicx}
\usepackage{color}

\usepackage{hyperref}

\usepackage{amsmath}

\usepackage{amssymb}

\usepackage{ulem}

\usepackage[mathscr]{eucal}

\def\dist{\mathop{\text{\normalfont dist}}}
\def\div{\operatorname {\text{div}}}

\newcommand{\R}{\mathbb{R}}

\newcommand{\N}{\mathbb{N}}

\newcommand{\ve}{\varepsilon}

\def\S{{\mathcal {S}_{G,H}}}
\def\R{{\mathbb {R}}}
\def\N{{\mathbb {N}}}
\def \div {\mathrm{div}}

\def \Leb {\mathscr{L}^N}

\newtheorem{theorem}{Theorem}[section]
\newtheorem{lemma}[theorem]{Lemma}
\newtheorem{proposition}[theorem]{Proposition}
\newtheorem{corollary}[theorem]{Corollary}
\theoremstyle{definition}

\newtheorem{definition}[theorem]{Definition}
\newtheorem{example}[theorem]{Example}
\theoremstyle{remark}
\newtheorem{remark}[theorem]{Remark}
\numberwithin{equation}{section}

\begin{document}

\title[Shape optimization in Orlicz-Sobolev spaces]{A constrained shape optimization problem in Orlicz-Sobolev spaces}

\author[J.V. da Silva, A.M. Salort, A. Silva, J.F. Spedaletti]{Jo\~{a}o Vitor da Silva, \,\,\,Ariel M. Salort, \,\,\, Anal\'ia Silva\\ $\&$\\ Juan F. Spedaletti}
\address[J.V da Silva]{Departamento de Matem\'atica FCEyN - Universidad de Buenos Aires and IMAS - CONICET. Ciudad Universitaria, Pabell\'on I (C1428EGA)
Av. Cantilo 2160. Buenos Aires, Argentina.}

\email{jdasilva@dm.uba.ar}
\urladdr{https://www.researchgate.net/profile/Joao\_Da\_Silva12}

\address[A.M. Salort]{Departamento de Matem\'atica FCEyN - Universidad de Buenos Aires and IMAS - CONICET. Ciudad Universitaria, Pabell\'on I (C1428EGA)
Av. Cantilo 2160. Buenos Aires, Argentina.}

\email{asalort@dm.uba.ar} \urladdr{http://mate.dm.uba.ar/~asalort/}
\address[A. Silva]{Departamento de Matem\'atica, Universidad Nacional de San Luis and IMASL - CONICET. Ej\'ercito de los Andes 950 (D5700HHW), San Luis, Argentina.}

\email{acsilva@unsl.edu.ar}
\urladdr{https://analiasilva.weebly.com}

\address[J.F. Spedaletti]{Departamento de Matem\'atica, Universidad Nacional de San Luis and IMASL - CONICET. Ej\'ercito de los Andes 950 (D5700HHW), San Luis, Argentina.}

\email{jfspedaletti@unsl.edu.ar}



\begin{abstract} In this manuscript we study the following  optimization problem: given a bounded and regular domain $\Omega \subset \R^N$ we look for an optimal shape for the ``$\mathrm{W}-$vanishing window'' on the boundary with prescribed measure over all admissible profiles in the framework of the Orlicz-Sobolev spaces associated to constant for the ``Sobolev trace embedding''. In this direction, we establish existence of minimizer profiles and optimal sets, as well as we obtain further properties for such extremals. Finally, we also place special emphasis on analyzing the corresponding optimization problem involving an ``$\mathrm{A}-$vanishing hole'' (inside the domain) with volume constraint.
\end{abstract}

\maketitle


\section{Introduction}

\subsection{A brief historic overview}
Shape optimization problems constitute an important landmark concerning the modern development of the mathematical theory of optimization. Such issues are a longstanding subject of investigation, and currently they yet deserve increased attention by the academic community due to their intrinsic connection with several pivotal questions coming from pure and applied sciences. Some enlightening examples of such issues appear in eigenvalue problems with geometric constraints, optimization problems with constrained perimeter or volume, optimal design problems, problems in structural optimization, free boundary optimization problems, just to mention a few.

Heuristically, a shape optimization problem can be mathematically written as follows:
$$
   \min \left\{\mathcal{J}(\mathcal{O}): \mathcal{O} \subset \Omega \,\,\,\text{with}\,\,\,\mathcal{O}\,\,\,\text{fulfilling a certain property}\,\,\,\mathbb{P}\right\},
$$
where $\Omega \subset \R^N$ is a bounded open set, $\mathcal{O}$ is an \textit{a priori} unknown configuration (in general satisfying a specific property related to some constraint) and $\mathcal{J}$ is a ``cost functional'', which in several situations has an explicit integral representation, whose link with the competing configuration $\mathcal{O}$ arises via a solution of a PDE (cf. \cite{BucBut}, \cite{HerPier} and \cite{SZ} for nice surveys with a number of illustrative examples, we also recommend the reading of \cite{AAC}, \cite{daSDelPR}, \cite{daSilRos}, \cite{BMW}, \cite{BondRosSped}, \cite{BondSped2}, \cite{BRW1},  \cite{Lederman96}, \cite{Mart}, \cite{OliTei}  and \cite{Tei3} for other references with regard to free boundary and shape optimization problems).

In the scope of the modern Analysis and PDE's theory, the \textit{Sobolev Trace Embedding Theorem}, namely
 $$
    W^{1, p}(\Omega) \hookrightarrow L^q(\partial \Omega)
 $$
 with the associated estimate for a constant $\mathcal{S}>0$ (\textit{Sobolev trace constant})
 $$
    \mathcal{S}\|u\|^p_{L^{q}(\partial \Omega)} \leq \|u\|^p_{W^{1, p}(\Omega)} \quad (\text{Sobolev trace inequality})
 $$
 constitutes a fundamental tool in order to study certain issues in mathematics such as eigenvalue and Steklov type problems, functional type inequalities, existence and solvability of boundary-value problems among others.

 Historically, optimization problems associated to the \textit{best constant for Sobolev trace embedding}, namely
 $$
   \mathcal{S}_{p, q} = \inf \left\{\frac{\int_{\Omega} |\nabla u|^p+ |u|^pdx}{\left(\int_{\partial  \Omega} |u|^q d \mathcal{H}^{N-1}\right)^{\frac{p}{q}}}: u \in W^{1, p}(\Omega)\setminus W_0^{1, p}(\Omega)\right\},
 $$
have received a warm attention by several authors in the last decades. The list of contributions is fairly diverse including aspects such as uniform bounds, symmetry properties, asymptotic behavior, local regularity of extremals and their free boundaries, approximations and many others (cf. \cite{Cianc}, \cite{DelPBN}, \cite{BLDR},  \cite{BOR2}, \cite{BRS1},  \cite{BRW1}, \cite{BRW2}, \cite{BS}, \cite{Rossi03} and \cite{Rossi05} for more details). Particularly, we must highlight that in \cite{DelPBN}, \cite{Den99}, \cite{BOR2}, \cite{BRS1} and \cite{BRW2}  the authors studied for the $p-$Laplacian operator the problem of finding an optimal hole/window into the domain (resp. on the boundary) with prescribed measure associate to best constant for the Sobolev trace embedding. More precisely, they analyze the following quantity:
\begin{equation}\tag{{\bf S.E.C.}}\label{eqpSEC}
   \displaystyle \mathcal{S}(\Gamma) = \inf_{u \in \mathrm{X}_{\Gamma}} \frac{\int_{\Omega} |\nabla u|^p+ |u|^pdx}{\left(\int_{\partial  \Omega} |u|^q d \mathcal{H}^{N-1}\right)^{\frac{p}{q}}},
\end{equation}
where $1\leq q < p_{\star} = \frac{p(N-1)}{N-p}$ (the critical exponent in the Sobolev trace embedding) and
 $$
   \mathrm{X}_{\Gamma} = \left\{u \in W^{1, p}(\Omega) \setminus W^{1,
   p}_0(\Omega);\,\,u=0 \,\,\Leb-\text{a.e.}\,\, \text{in}\,\,\Gamma\,\,\,(\text{resp.} \,\,\mathcal{H}^{N-1}\,\, \text{a.e.\,\,on}\,\,\,\partial \Gamma)\right\},
 $$
where $\Leb$ (resp. $\mathcal{H}^{N-1}$) stands for the $N-$dimensional Lebesgue measure (resp. $(N-1)-$dimensional Hausdorff measure). Furthermore, another important issue in these works regards to the following shape optimization problem: for any fixed $0<\alpha<1$ the optimization problem
{\small{
 $$
   \mathcal{S}(\alpha) = \inf\left\{\mathcal{S}(\Gamma): \Gamma \subset \Omega \,\,\,(\text{resp.} \,\, \Gamma \subset \partial \Omega)\,\text{s.t.}\, \frac{\Leb(\Gamma)}{\Leb(\Omega)}=\alpha \left(\text{resp.}\,\,\frac{\mathcal{H}^{N-1}(\Gamma)}{\mathcal{H}^{N-1}(\partial\Omega)}=\alpha \right)\right\}
 $$}}
is achieved by a pair $(u_0, \Gamma_0)$ (an existence result). Moreover, under suitable regularity assumptions on the boundary, they obtain that $\Gamma_0 = \{u=0\}$ (an explicit characterization result).

In the same way that in the classical Sobolev spaces, such trace embedding also plays a significant role in more general contexts governed by spaces with non-standard growth, for which naturally we can quote the well-known \textit{Orlicz-Sobolev spaces} (cf. \cite{Adams77}, \cite{Cianchi96}, \cite{Cianc} and \cite{DonTrun} for such subjects). Such spaces extend the classical notion of Sobolev spaces to a context with non-power nonlinearities (cf. \cite{BirOrl} and \cite{Orl}), and currently such spaces are fully understood and studied in Analysis, PDE's, Free boundary problems, etc (cf. \cite{AdamsFour}, \cite{FucJohnKuf}, \cite{KrasnRut} and \cite{RaoRen} for some surveys, and \cite{BondSal}, \cite{Lieberman}, \cite{Mart}, \cite{MW} and \cite{Mon} for some applications in PDE's theory).

According to our knowledge, up to the date, there is no research concerning such optimization problems \eqref{eqpSEC} in general sceneries with non-standard growth. For this very reason, such lack of investigations was one of our main starting points in considering shape optimization problems associate to the Sobolev trace embedding in the framework of Orlicz-Sobolev spaces.

\subsection{Statement of main results}

For the functional implementations in this article, we define the
\textit{Orlicz-Sobolev embedding constant} as follows:
\begin{equation}\tag{{\bf O.S.E.C.}}\label{eqSEC}
\displaystyle S_{G,H}:=\inf_{u \in X} \frac{\int_\Omega G(|\nabla u|)+\ G(|u|)dx}{\int_{\partial  \Omega} H(|u|) d \mathcal{H}^{N-1}},
\end{equation}
where $X=W^{1,G}(\Omega)\setminus W_0^{1,G}(\Omega)$ is an admissible functional class defined under Orlicz-Sobolev spaces, and $G$ and $H$ (throughout this manuscript) are suitable Young functions, both to be clarified \textit{a posteriori}, see Section \ref{Prelim} for more details. It is worth highlighting that such a quantity \eqref{eqSEC} is linked to some extent with the compact trace embedding
$$
   W^{1, G}(\Omega) \hookrightarrow L^H(\partial \Omega),
$$
where $H$ and $G$ fulfill a certain compatibility condition, see Theorem \ref{compacidad} for   details.

Different from the Rayleigh quotient in \eqref{eqpSEC}, our
definition employs an inhomogeneous quotient. This imposes an extra
difficulty in our problem, which will be overcame by asking a
``normalization'' of boundary term  (the associated modular),
and then, we consider the class of admissible functions subject to
such a constraint. Furthermore, we point out that different from its
$p-$power counterpart (cf. \cite{DelPBN}, \cite{BRS1} and
\cite{BRW2}), this version involves further extensions and
difficulties that are treated and resolved throughout this
manuscript.

In face of the previous statements, the first purpose of our manuscript consists in analyzing the shape optimization problem related to the ``analogue'' trace embedding constant associated to the Orlicz-Sobolev spaces. In this direction, we consider a regular and bounded domain $\Omega\subset \R^N$ and a subset of the boundary $\mathrm{W}\subset \partial \Omega$  (a ``window'') such that $\mathrm{W}\neq \partial \Omega$. Thus, we define the minimization problem
\begin{equation}\tag{{\bf Min}}\label{eqMin}
S_{G,H}(\mathrm{W}):=\inf\left\{\int_\Omega G(|\nabla u|)+ G(|u|)dx\colon\int_{\partial  \Omega} H(|u|) d \mathcal{H}^{N-1}=1\right\},
\end{equation}
where the infimum is taken over the set
$$
   \mathrm{X}_{\mathrm{W}}:=\{u\in \mathrm{X}\colon u=0\,\,\, \mathcal{H}^{N-1}-\text{a.e. on } \mathrm{W}\}.
$$
In our researches, the constant $S_{G,H}(\mathrm{W})$ represents the counterpart for the first \textit{H-Steklov eigenvalue} of the ``$G-$Laplacian operator'', which is defined as
$$
   \Delta_G u : = \div\left(\frac{G^{\prime}(|\nabla u|)}{|\nabla u|}\cdot \nabla u\right).
$$
Furthermore, notice that if $G(t)=H(t) = \frac{1}{p}t^{p}$ for $p>1$, then we fall into the well-known case of the Steklov eigenvalue for the $p-$Laplacian operator.

Next, let $0<\alpha<\mathcal{H}^{N-1}(\partial \Omega)$ be a fixed constant. Taking into account \eqref{eqMin}, we define the following shape optimization problem:
\begin{equation}\tag{{\bf $\alpha-$Window}}\label{shapeOrlicz}
  \S(\alpha):=\inf\{S_{G,H}(W)\colon W\subset \partial\Omega\text{ and }\mathcal{H}^{N-1}(W)=\alpha\}.
\end{equation}
In this framework, a set $W\subset \partial\Omega$ in which the above infimum is achieved so-called \textit{optimal window} for the constant $\S(\alpha)$.

Our first result provides the existence of minimizers and optimal shapes for our optimization problem, with a lower bound estimate for the null set of minimizers.

\begin{theorem}\label{Thm1}
Let $0<\alpha<\mathcal{H}^{N-1}(\partial \Omega)$ be a fixed quantity. Then
\begin{enumerate}
\item\label{Stat1} There exists a function $u_0\in \mathrm{X}$ with $
\mathcal{H}^{N-1}(\{u_0=0\})\geq\alpha
$
such that
$$
\S(\alpha)=\int_\Omega G(|\nabla u_0|)+ G(|u_0|)dx.
$$
\item\label{Stat2} There exists a set $\mathrm{W}_0\subset\partial\Omega$ such that
$$
\mathcal{H}^{N-1}(\mathrm{W}_0)=\alpha \quad \text{and} \quad \S(\alpha) = S_{G,H}(\mathrm{W}_0).
$$

\end{enumerate}
\end{theorem}


In contrast with previous result, now we establish that minimizers have an $\alpha-$sharp measure provided we assume enough regularity on the boundary.

\begin{theorem}\label{ThmCharac}
Suppose that the assumptions from Theorem \ref{Thm1} are in force. If $\partial \Omega$ is  an enough regular set, let us say $C^{1,\beta}$, then $\mathcal{H}^{N-1}(\{u_0=0\})=\alpha$.
\end{theorem}
It is worth to mention that any optimization pair $(u,
\mathrm{W})$ coming from Theorem \ref{Thm1} is linked thought the
following singular/degenerate elliptic PDE:
$$
  \left\{
\begin{array}{rclcl}
  -\div\left(\frac{G^{\prime}(|\nabla u|)}{|\nabla u|}\cdot \nabla u\right) + \frac{g(|u|)u}{|u|} & = & 0 & \text{in} & \Omega \setminus \mathrm{W}\\
  g(|\nabla u|)\frac{\nabla u}{|\nabla u|}\cdot\eta & = &  S_{G,H}(\mathrm{W})
\frac{h(|u|)u}{|u|} & \text{on} & \partial \Omega\setminus\mathrm{W} \\
  u & = & 0 & \text{on} & \mathrm{W}.
\end{array}
\right.
$$

In comparison to previous results in the literature, namely \cite{DelPBN}, \cite{BRW1} and  \cite{BRW2}, our results include the case where the PDE/Neumann condition become ``singular''.

In final part of article, we put special attention to the shape optimization problem for finding an optimal interior hole
$A\subseteq\Omega$ with prescribed volume associated to the
``Orlicz-Sobolev embedding'' constant, i.e.,
\begin{equation}\label{EqMinHole}
  \mathbb{S}_{G,H}(\mathrm{A}) := \inf_{}\left\{\int_\Omega G(|\nabla u|)+ G(|u|)dx\colon \int_{\partial\Omega}  H(|u|) d \mathcal{H}^{N-1}=1\right\},
\end{equation}
where the infimum is taken in the class
$$
 \mathrm{X}_{\mathrm{A}}:=\left\{u\in W^{1,G}(\Omega)\setminus W_0^{1,G}(\Omega) \colon u=0\,\,\,\text{a.e. in } \mathrm{A}\right\}.
$$
In the same way we can consider an optimal design problem associated to the constant $\mathbb{S}_{G,H}(A)$, as follows: for $\alpha\in (0, \Leb(\Omega))$ we define
\begin{equation}\tag{{\bf $\alpha-$Hole}}\label{OptHole}
\mathbb{S}(\alpha) = \inf\left\{\mathbb{S}_{G,H}(\mathrm{A}):
\mathrm{A} \subset \Omega \quad \text{and} \quad  \Leb(\mathrm{A})=\alpha
\right\}.
\end{equation}
A set $A\subset\Omega$ in which the above infimium is achieved is called \textit{optimal interior hole}.

The following result is the analogous one of Theorems \ref{Thm1} and \ref{ThmCharac} for the optimization of a hole into the domain instead on the boundary. The case of the $p-$Laplacian operator with Steklov boundary condition it was considered in \cite{BRW2}.
\begin{theorem}\label{Thm4}
Given $0<\alpha<\Leb(\Omega)$. There exists a set $\mathrm{A}_0\subset \Omega$ such that $\Leb(\mathrm{A}_0)=\alpha$ and $\mathbb{S}_{G, H}(\mathrm{A}_0) = \mathbb{S}(\alpha)$. Moreover, every corresponding extremal $u_0$ to \eqref{OptHole} verifies that $\Leb(\{u_0=0\})=\alpha$.
\end{theorem}

In the next result we prove that there is no upper bound for $\mathbb{S}_{G, H}(\mathrm{A})$, where $\mathrm{A} \subset \Omega$ is an optimal interior hole. 

\begin{theorem}\label{Thm5}
Let $0<\alpha<\Leb(\Omega)$ be a fixed quantity. Then, the following statement holds true:
$$
 \sup\{ \mathbb{S}_{G, H}(\mathrm{A}): \mathrm{A} \subset \Omega \,\,\, \text{and}\,\,\,\Leb(\mathrm{A})= \alpha\}=+\infty.
$$
\end{theorem}
Next, in order to give sense to ``Orlicz-Sobolev trace constant'' for functions vanishing in a negligible subset (zero $N-$dimentional Lebesgue measure) we will need to consider the space
$$
   W_{\mathrm{A}}^{1,G}(\Omega)=\overline{C_0^\infty (\overline{\Omega}\setminus
\mathrm{A})}
$$
where the closure is taken in $W^{1,G}-$norm, i.e., $W_{\mathrm{A}}^{1,G}(\Omega)$ are the functions that can be approximated by smooth functions that vanish in a neighborhood of $\mathrm{A}$ (compare with \cite[Theorem 7.1.7]{FucJohnKuf}).

In this context the ``Orlicz-Sobolev constant'' is defined as
$$
\mathbb{S}_{\mathrm{A}}=\inf_{W_A^{1,G}(\Omega)}\left\{\int_\Omega G(|u|) +  G(|\nabla u|)dx\colon\int_{\partial  \Omega} H(|u|) d \mathcal{H}^{N-1}=1\right\}.
$$

At this point, it is important  to question when $\mathbb{S}_{\mathrm{A}}$ recovers the usual ``Orlicz-Sobolev trace constant'', i.e., when $\mathbb{S}_{\mathrm{A}}=\mathbb{S}_{\emptyset}$. A key ingredient for this result is the notion of $G-$capacitary sets (see section \ref{Sec6} for more details). We prove the following necessary and sufficient condition to this to hold.

\begin{theorem}\label{Thm1.6} $\mathbb{S}_{\mathrm{A}} = \mathbb{S}_\emptyset$ if and only if $\text{Cap}_G(\mathrm{A})=0$.
\end{theorem}

Next, we address   the continuity of  $\mathbb{S}_{\mathrm{A}}$ with respect to $\mathrm{A}$ in the Hausdorff topology. Furthermore, we establish the continuity of corresponding extremals (in the $W^{1, G}$ norm) with respect to the Hausdorff topology of the sets $\mathrm{A}$.

\begin{theorem}\label{Thm1.7}
Let $\mathrm{A},\mathrm{A_k}$ be closed sets such that
$$
  \text{dist}_{\mathcal{H}}(\mathrm{A}, \mathrm{A_k})\to 0 \quad  \text{as} \quad k\to\infty.
$$
Then,
$$
  |\mathbb{S}_{\mathrm{A_k}}-\mathbb{S}_{\mathrm{A}}|\to 0 \quad \text{as} \quad k\to\infty.
$$
Moreover, if $u_k$ is an extremal for
$\mathbb{S}_{\mathrm{A_k}}$ normalized such that $\int_{\partial
\Omega} H(|u_k|) d \mathcal{H}^{N-1}=1$, then up to a
subsequence,
$$
  u_{k} \rightarrow u \quad \text{strongly in} \quad W^{1,G}_A(\Omega)
$$
and $u$ is an extremal for $\mathbb{S}_{\mathrm{A}}$.
\end{theorem}

In conclusion, a natural issue is what can be inferred about the
extremals $u$ and ``the optimal set'' $\{u = 0\} \subset \partial
\Omega$ when the domain has certain symmetry. In our last result, we
prove that (when $\Omega$ is a unity ball) there exists an extremal
(resp. an optimal window) spherically symmetric.

\begin{theorem}\label{SymmetThm}
Let $\Omega = B_1$ and let $0<\alpha<\mathcal{H}^{N-1}(\partial \Omega)$ fixed. Then, there exists an optimal window which is a spherical cap.
\end{theorem}

\subsection*{Organization of the paper} Our manuscript is organized as follows: in Section \ref{Prelim} we collect some preliminary results in the framework of Orlicz-Sobolev spaces. In Section \ref{PropMin} we
present a number of properties for minimizers of our optimization
problem. Section \ref{OptProb} is devoted to analyze our shape
optimization problem and its features. In Sections  \ref{OptProbHole} and \ref{Sec6} we establish existence and further results for extremals and optimal sets for the corresponding $\mathrm{A}-$vanishing optimization problem. Finally, Section \ref{SpheSym} is dedicated to prove a spherical symmetrization result.

\section{Technical tools}\label{Prelim}
In this section we introduce some well-known definitions and auxiliary results.
\subsection{Young functions}
We  consider the well-known set of \textit{Young functions}. A function $G:\R^+\to \R$ in this class admits the following  representation
$$
  G(t)=\int_0^t g(s)\,ds, \qquad t\geq 0,
$$
where $g: [0,\infty) \to [0,\infty)$ has the following properties:
\begin{itemize}
\item[(i)] $g(0)=0$,
\item[(ii)] $g(s)>0$ for $s>0$,
\item[(iii)] $g$ is right continuous at any point $s\geq 0$,
\item[(iv)] $g$ is nondecreasing on $(0,\infty)$.
\end{itemize}

The following lemma provides several useful properties on Young functions.

\begin{lemma} \cite[Lemma 3.2.2]{FucJohnKuf}.
A Young function $G$ is continuous, nonnegative, strictly increasing
and convex on $[0,\infty)$. Moreover,
\begin{itemize}
\item[(i)] $\displaystyle G(0)=0 \quad \text{and}\quad \lim_{t\to+\infty} G(t)=\infty$;
\item[(ii)]  $\displaystyle\lim_{t\to 0^+} \frac{G(t)}{t}=0 \quad \text{and}\quad \lim_{t\to +\infty} \frac{G(t)}{t}=\infty$.
\end{itemize}
\end{lemma}

For our purposes we consider Young functions satisfying the so-called $\Delta_2-$condition.
\begin{definition}
A Young function $G$ satisfies the $\Delta_2$ condition (or doubling condition) if
$$
  G(2t)\leq C G(t)
$$
for all $t\geq 0$ for a fixed positive constant $C$.
\end{definition}
In particular, (cf. \cite[Theorem 3.4.4]{FucJohnKuf}) a Young function $G$ satisfies the $\Delta_2-$condition if and only if
$$
\limsup_{t \to \infty} \frac{t G^{\prime}(t)}{G(t)} <\infty.
$$
It is worth mentioning that such a kind of growing condition appears naturally when studying Orlicz-Sobolev spaces. Furthermore, we must compare such a condition with one considered in the Lieberman's pioneering work \cite{Lieberman},
\begin{equation}\label{Liebcond}
  g^{-}-1\leq \frac{tg^{\prime}(t)}{g(t)} \leq g^+-1   \quad \forall\,\, t>0,
\end{equation}
for constants $0<g^{-}\leq g^+<\infty$, which establishes regularity estimates for weak solutions in Orlicz-Sobolev spaces. In fact, it is straightforward to see that such a condition implies that
$$
   g^{-} \leq \frac{tG^{\prime}(t)}{G(t)} \leq g^+  \quad \forall\,\, t>0.
$$

Finally, the following version of the triangle inequality for Young functions holds.

\begin{lemma}[{\cite[Lemma 2.6]{BondSal}}]\label{Lemma2.3}
Let $G$ be a Young function. Then for every $\eta>0$ there exists $C_\eta>0$ such that
$$
G(a+b)\leq C_\eta G(a)+ (1+\eta)^{g^+}G(b)  \quad a,b>0.
$$
\end{lemma}

\begin{example} Some well-known examples of Young functions include the following:
\begin{enumerate}
  \item $G(t) = t^p$ for $p>1$;
  \item $G(t) = t^p(\mathfrak{a}|\log t|+\mathfrak{b})$ for $p>1$ and $\mathfrak{a}, \mathfrak{b}>0$;
  \item $G(t) = \frac{t^p}{\mathfrak{a}\log (t+e)+\mathfrak{b}}$ for $p>1$ and $\mathfrak{a}, \mathfrak{b}>0$;
  \item If $G_1$ and $G_2$ are Young functions, then the composition $G(t) = (G_1\circ G_2)(t)$  is also a Young function;
\item Finite linear combinations of Young functions with non-negative coefficients are also Young functions;
\item The maximum among a finite family of Young functions is also a Young function.
\end{enumerate}
\end{example}

\subsection{Orlicz-Sobolev spaces and traces}

Given a Young function $G$ and a bounded open set $\Omega$ we consider the spaces $L^G(\Omega)$ and $W^{1,G}(\Omega)$ defined as follows:
\begin{align*}
L^G(\Omega) & :=\{ u : \R\to \R \text{ measurable such that }\Phi_{G,\Omega}(u)<\infty\},\\
W^{1,G}(\Omega)& :=\{ u \in L^{G}(\Omega) \text{ such that }
\Phi_{G,\Omega}(|\nabla u|) <\infty\},
\end{align*}
where $\nabla u$ is considered in the distributional sense and the
\textit{modular} $\Phi_{G,\Omega}$ stands for
$$
\Phi_{G,\Omega}(u)= \int_\Omega G(|u|)\,dx.
$$
These spaces are endowed with the so-called \textit{Luxemburg norm} defined
as follows
$$
\|u\|_{L^G(\Omega)}=\inf \left\{ \lambda>0 : \Phi_{G, \Omega}\left(
\frac{u}{\lambda}\right)\leq 1 \right\}
$$
and
$$
\|u\|_{W^{1,G}(\Omega)} = \|u\|_{L^G(\Omega)} + \|\nabla
u\|_{L^G(\Omega)}.
$$

It is worth highlighting that the spaces $L^G(\Omega)$ and
$W^{1,G}(\Omega)$ are reflexive and separable Banach spaces if and only if $G$ and $\tilde{G}$ satisfy the $\Delta_2-$condition (cf. \cite[Theorem 8.20]{AdamsFour}, \cite[Theorem 3.13.9]{FucJohnKuf} and  \cite[page
226]{KrasnRut}, where
$$
  \displaystyle \tilde{G}(t): = \sup_{s\geq0}\{st-G(s)\}
$$
denotes the complementary function (or Young conjugate) to $G$ (cf. \cite{Adams77, Cianchi96, Cianc, DonTrun}). From now on, we will also assume that $\tilde{G}$ satisfies the $\Delta_2-$condition (cf. \cite[Theorem 3.4.7]{FucJohnKuf} and \cite[page 5]{KrasnRut}).

Now, we introduce the notion of weak solution.
\begin{definition}\label{soludebil}
A function $u\in X_W$ is said to be a \textit{weak solution} to
\begin{equation}\label{MainEq}
  \left\{
\begin{array}{rclcl}
  -\div\left(\frac{G^{\prime}(|\nabla u|)}{|\nabla u|}\cdot \nabla u\right) + \frac{g(u)u}{|u|} & = & 0 & \text{in} & \Omega \setminus \{u=0\}\\
  g(|\nabla u|)\frac{\nabla u}{|\nabla u|}\cdot\eta & = &  S_{G,H}(\mathrm{W})
\frac{h(u)u}{|u|} & \text{on} & \partial \Omega\setminus\mathrm{W} \\
  u & = & 0 & \text{on} & \mathrm{W},
\end{array}
\right.
\end{equation}
if
$$
\int_\Omega g(|\nabla u|)\frac{\nabla u}{|\nabla u|}\cdot\nabla \phi + \frac{g(u)u\phi}{|u|}\,dx=S_{G,H}(\mathrm{W})
\int_{\partial \Omega} \frac{h(u)u\phi}{|u|}\,d\mathcal{H}^{N-1},
$$
for every test function $\phi\in C^{\infty}(\Omega)$. Here $\nabla u \cdot\eta$ is the outer unit normal derivative.
\end{definition}

\begin{remark}\label{remarkk}
The following statements hold true:

\begin{enumerate}
  \item From the available regularity theory from \cite{Kor} and \cite[Ch.5]{Lieberman}, any minimizer $u_0$ of \eqref{eqMin} fulfills that $u_0 \in C_{\text{loc}}^{1,\beta}(\Omega)$ for some $0 < \beta <1$.
  \item From \cite{Lieberman88}, if $\partial \Omega \setminus \overline{\mathrm{W}} \in C^{1, \gamma}$, then $u \in C^{1, \gamma}(\overline{\Omega}\setminus \mathrm{W})$ up to the boundary, for some $0<\gamma<1$.
  \item Notice that, if $u_0$ is minimizer to \eqref{eqMin}, then  $|u_0|$ as well. Hence, $v = |u_0|$ is a weak solution of
      $$
\left\{
\begin{array}{rclcl}
  -\div\left(\frac{G^{\prime}(|\nabla v|)}{|\nabla v|}\cdot \nabla v\right) + g(v) & = & 0 & \text{in} & \Omega \setminus \{v=0\}\\
  g(|\nabla v|)\frac{\nabla v}{|\nabla v|}\cdot\eta & = &  S_{G,H}(\mathrm{W})
h(v) & \text{on} & \partial \Omega\setminus\mathrm{W} \\
  v & = & 0 & \text{on} & \mathrm{W}.
\end{array}
\right.
$$
Thus, by using the maximum principle (cf. \cite{Mon}), we obtain that $u_0$ does not change its sign. Consequently, we can always assume that
$$
  u_0>0 \quad \text{in} \quad \Omega \qquad \text{and} \qquad u_0\geq0 \quad \text{on} \quad \partial \Omega.
$$
\item From Hopf's Lemma (cf. \cite[Theorem 1]{MiSh}) and boundary regularity results in \cite{Lieberman88} we obtain that any nonnegative solution $u_0$ to \eqref{MainEq} satisfies
    $$
    u_0>0 \qquad \text{in} \qquad \overline{\Omega} \setminus \overline{\mathrm{W}}.
    $$
\end{enumerate}
\end{remark}

From now on, we will assume the following conditions on the Young function $G$:
\begin{equation}\label{Condcompac}
\int_0^1 \frac{G^{-1}(s)}{s^{1+\frac{1}{N}}}\, ds<\infty \text{ and
} \int_{1}^{\infty} \frac{G^{-1}(s)}{s^{1+\frac{1}{N}}}\,ds=\infty,
\end{equation}
which enable us to access to a compactness result regarding traces. For this purpose
we recall that the Orlicz-Sobolev conjugate of $G$ (see
\cite[p. 352]{FucJohnKuf}) is defined as follows
$$
(G^{\ast})^{-1}(t)=\int_0^t\frac{(G)^{-1}(s)}{s^{1+\frac{1}{N}}}\, ds.
$$

We say that $\mathfrak{A}$ increases more slowly than $\mathfrak{B}$ and   denote
$\mathfrak{A}<<\mathfrak{B}$ if and only if
$$
 \displaystyle \lim_{t\to \infty} \frac{\mathfrak{A}(t)}{\mathfrak{B}(\lambda t)}=0
$$
for any $\lambda>0$.

Finally, the following result give us the compactness of the trace in the Orlicz-Sobolev space $W^{1,G}(\Omega)$.

\begin{theorem}[{\cite[Theorem 7.4.6]{FucJohnKuf}}]\label{compacidad}
Let $\Omega\subset \R^N$, $N\geq 2$, be a $C^{0,1}$ domain. Let $G$
be a Young function  satisfying the conditions
\eqref{Condcompac}, then the embedding
$$
  W^{1,G}(\Omega)\hookrightarrow L^H(\partial \Omega)
$$
is compact for every Young function $H$ such that $H<<\Psi$, where
$$
\Psi(t)=(G^{\ast}(t))^{\frac{N-1}{N}}.
$$
\end{theorem}

\section{Minimizers and optimal shapes}\label{PropMin}

This section will be devoted to establish   existence of extremals and optimal shapes for our optimization problem.

From now on, a set $\Omega\subset \R^N$ for $N\geq 2$, will denote a
$C^{0,1}$ open bounded domain. The following result shows existence
of minimizer for $S_{G,H}(\mathrm{W})$.

\begin{theorem}[{\bf Existence of minimizers}]\label{ExistThm}
The constant $S_{G,H}(\mathrm{W})$ in \eqref{eqMin} is achieved for some function $u_0\in \mathrm{X}_{\mathrm{W}}$.
\end{theorem}

\begin{proof}
The proof follows as a consequence of the direct method in the calculus of variations. Indeed, take a minimizing sequence $\{u_k\}_{k\in \N}$ of $S_{G,H}(\mathrm{W})$. That is
$$
S_{G,H}(\mathrm{W})=\lim_{k\to\infty}\Phi_{G,\Omega}(|\nabla u_k|)+\Phi_{G,\Omega}(u_k),
$$
$\displaystyle \int_{\partial\Omega}H(|u_k|)d\mathcal{H}^{N-1}=1$, and $u_k=0$ $\mathcal{H}^{N-1}-$ a.e. on $\mathrm{W}$ for all $k\in \N$. The above limit tell us that there exists a positive constant $C$ such that
$$
\|u_k\|_{W^{1,G}(\Omega)}\leq C,\quad \forall k\in \N.
$$
Now, from the reflexivity of the space $W^{1,G}(\Omega)$, and by   Theorem \ref{compacidad} there exists a function $u\in W^{1,G}(\Omega)$ such that, up to a subsequence,
\begin{align}\label{debil}
&u_k\rightharpoonup u_0 \text{ weakly  in }   W^{1,G}(\Omega),\\
&\label{fuerteborde1}
u_k\to u_0 \text{ strongly  in }L^{H}(\partial \Omega),\\
&\label{ctp1}
u_k\to u_0\text{ a.e. in }\partial \Omega \text{ (see \cite[p. 200]{FucJohnKuf} for details)}.
\end{align}
Note that \eqref{fuerteborde1} implies (see Remark \eqref{equivmodulares})
$$
\int_{\partial\Omega}H(|u_0|)d\mathcal{H}^{N-1}=1.
$$
Observe that the above implies that $u_0$ can not be zero over the complete boundary $\partial\Omega$, in this way $u_0\in W^{1,G}(\Omega)\setminus W_0^{1,G}(\Omega)$. Moreover, by \eqref{ctp1} we get that $u_0=0\, \mathcal{H}-$a.e on $\mathrm{W}$. Then,
$$
S_{G,H}(\mathrm{W})\leq \Phi_{G,\Omega}(|\nabla u_0|)+\Phi_{G,\Omega}(u_0).
$$
Finally, from the convexity of $G$, the weakly lower semicontinuity of the application $v \mapsto \Phi_{G,\Omega} (|\nabla v|)+\Phi_{G,\Omega}(v)$ and the Fatou's lemma (cf. \cite[Theorems 2.1.17, 2.2.8 and Lemma 2.3.16]{DHHR}) we get
$$
\begin{array}{rcl}
  \Phi_{G,\Omega} (|\nabla u_0|)+ \Phi_{G,\Omega}(u_0) & \leq & \displaystyle \liminf_{k\to\infty} \left[\Phi_{G,\Omega} (|\nabla u_k|)+
\Phi_{G,\Omega}(u_k)\right] \\
   & = & S_{G,H}(\mathrm{W}).
\end{array}
$$
This proves the theorem.
\end{proof}

\begin{remark}\label{equivmodulares}
From \cite[Lemma 2.1.14]{DHHR} we know that
$\Phi_{H,\partial\Omega}(u_k)=1$ if and only if
$\|u_k\|_{L^H(\partial\Omega)}=1$ for all $k$. Since
$\lim_{k\to\infty}\|u_k\|_{L^{H}(\partial\Omega)}=\|u_0\|_{L^H(\partial\Omega)},$
we obtain that $\|u_0\|_{L^H(\partial\Omega)}=1$, and one more time
using \cite[Lemma 2.1.14]{DHHR} we conclude that
$\Phi_{H,\partial\Omega}(u_0)=1$.
\end{remark}

By following the approach from \cite[Theorem 1.1]{BRW2}, we will prove that $S_{G, H}(\mathrm{W})$ is a lower semi-continuous map with respect to the window. As a result, we obtain   existence of an optimal shape (window) for our optimization problem.

\begin{theorem}[{\bf Existence of optimal shapes}] Let $(\mathrm{W_{\tau}})_{\tau>0} \subset \partial \Omega$ be a family of positive $\mathcal{H}^{N-1}-$measurable subsets and $\mathrm{W}_0 \subset \partial \Omega$ be a positive $\mathcal{H}^{N-1}-$measurable set, such that
$$
\chi_{W_{\tau}} \rightharpoonup \chi_{W_0} \quad \ast-\text{weakly in}\,\,\,L^{\infty}(\partial \Omega).
$$
Then,
$$
  \displaystyle S_{G, H}(\mathrm{W}_0) \leq \liminf_{\tau \to 0+} S_{G, H}(\mathrm{W}_{\tau}).
$$
\end{theorem}

\begin{proof} Let $(\mathrm{W_k})_{k \in \mathbb{N}} \subset (\mathrm{W_{\tau}})_{\tau>0}$ be a subsequence such that
$$
\displaystyle \mathcal{W} = \liminf_{\tau \to 0+} S_{G, H}(\mathrm{W}_{\tau}) = \lim_{k \to +\infty} S_{G, H}(\mathrm{W}_k).
$$
Now, for each $k \in \mathbb{N}$, we consider $u_k \in \mathrm{X}_{\mathrm{W}_k}$ a non-negative minimizer of $S_{G, H}(\mathrm{W}_k)$. Consequently, $(u_k)_{k \in \mathbb{N}}$ is bounded in $W^{1, G}(\Omega)$.
Now, from reflexivity of the space $W^{1,G}(\Omega)$  and Theorem \eqref{compacidad}, there exists a function $u_0\in W^{1,G}(\Omega)$ such that, up to a subsequence,
\begin{align}\label{debil1}
&u_k\rightharpoonup u_0 \text{ weakly  in }   W^{1,G}(\Omega),\\
&\label{fuerteborde}
u_k\to u_0 \text{ strongly  in }L^{H}(\partial \Omega),\\
&\label{ctp}
u_k\to u_0\text{ a.e. in }\partial \Omega \text{ (See \cite[Pag. 200]{FucJohnKuf} for details)}.
\end{align}
Particularly, $u_0$ is a non-negative profile with $\Phi_{H, \partial \Omega}(u_0) =1$ and using \cite[Theorems 2.1.17, 2.2.8 and Lemma 2.3.16]{DHHR}
$$
\Phi_{G,\Omega} (|\nabla u_0|)+ \Phi_{G,\Omega}(u_0) \leq
\liminf_{k\to\infty} \left[\Phi_{G,\Omega} (|\nabla u_k|)+
\Phi_{G,\Omega}(u_k)\right].
$$
Furthermore, for each $k \in \mathbb{N}$, $u_k = 0 \,\,\,\mathcal{H}^{N-1}-$a.e. on $\mathrm{W}_k$. From the  convergence
$$
\chi_{W_{\tau}} \rightharpoonup \chi_{\mathrm{W}_0} \quad \ast-\text{weakly in}\,\,\,L^{\infty}(\partial \Omega)
$$
and  \eqref{ctp} we conclude that
$$
\displaystyle \int_{\mathrm{W}_0} u_0(x) d\mathcal{H}^{N-1} = \lim_{k \to +\infty} \int_{\mathrm{W}_k} u_k(x) d\mathcal{H}^{N-1} = 0.
$$
Since $u_0$ is a non-negative profile we conclude that $u_0 = 0 \,\,\, \,\,\,\mathcal{H}^{N-1}-$a.e on $\mathrm{W}_0$. Therefore, $u_0$ is an admissible profile in the characterization of $S(\mathrm{W}_0)$ and
\begin{align*}
S_{G, H}(\mathrm{W}_0) & \leq   \Phi_{G,\Omega} (|\nabla u_0|)+ \Phi_{G,\Omega}(u_0)\\
&\leq \displaystyle \liminf_{k\to\infty} \left[\Phi_{G,\Omega} (|\nabla u_k|)+
\Phi_{G,\Omega}(u_k)\right]  \\
&=  \mathcal{W},
\end{align*}
which finishes the proof.
\end{proof}

\begin{remark}
The pair $(u_0, S_{G,H}(\mathrm{W}))$ from Theorem
\ref{ExistThm} is a solution, in the sense of   Definition \ref{soludebil}, of the following boundary value problem:
$$
\left\{
\begin{array}{rclcl}
  -\div\left(\frac{G^{\prime}(|\nabla u_0|)}{|\nabla u_0|}\cdot \nabla u_0\right) + \frac{g(u_0)u_0}{|u_0|} & = & 0 & \text{in} & \Omega \setminus \{u_0 = 0\}\\
  g(|\nabla u_0|) \frac{\nabla u_0}{|\nabla u_0|}\cdot\eta & = &  S_{G,H}(\mathrm{W})
\frac{h(u_0)u_0}{|u_0|} & \text{on} & \partial \Omega\setminus\mathrm{W} \\
  u & = & 0 & \text{on} & \mathrm{W},
\end{array}
\right.
$$
where $G^{\prime}(t)=g(t)$ and $H^{\prime}(t)=h(t)$.

Finally, observe that when $G(t)=H(t)=\frac{t^{p}}{p}$ with $p>1$, the last equation becomes the usual Steklov eigenvalue problem.
\end{remark}

\section{The shape optimization problem: optimal boundary window}\label{OptProb}

This section is devoted to provide the proofs of Theorem \ref{Thm1}
and Theorem \ref{ThmCharac}.

Taking into account the shape optimization problem \eqref{shapeOrlicz} the next result provides a characterization for the constant $\S(\alpha)$.

\begin{lemma}\label{lemma.caract} The following characterization holds true
$$
  \S(\alpha):= \inf\left \{
 \Phi_{G,\Omega}(|\nabla u|) +\Phi_{G,\Omega}(u) \colon
\Phi_{H,\partial\Omega}(u)=1,
\mathcal{H}^{N-1}(\{u=0\})\geq\alpha\right\},
$$
where the infimum is taken for functions in $W^{1,G}(\Omega)\setminus W_0^{1,G}(\Omega)$.
\end{lemma}
\begin{proof}
Let us define
$$
\widetilde{\S}(\alpha):= \inf\left \{
\Phi_{G,\Omega}(|\nabla u|) +\Phi_{G,\Omega}(u) \colon
\Phi_{H,\partial\Omega}(u)=1,
\mathcal{H}^{N-1}(\{u=0\})\geq\alpha\right\}.
$$
Under such a definition, we will prove that $\widetilde{\S}(\alpha)$ is equivalent to \eqref{eqMin}.

Firstly, we will prove that $\widetilde{\S}(\alpha)\leq\S(\alpha)$. Let $\mathrm{W}\subseteq\partial\Omega$ be such that $\mathcal{H}^{N-1}(\mathrm{W})\geq\alpha
$ and let $u\in \mathrm{X}_{\mathrm{W}}$ be a nonnegative extremal of $S_{G,H}(\mathrm{W})$. It
is easy to verify that $u$ is an admissible function in
$\widetilde{\S}(\alpha)$, so
$$\widetilde{\S}(\alpha)\leq\Phi_{G,\Omega}(|\nabla u|) +\Phi_{G,\Omega}(u)
=\mathcal{S}_{G,H}(\mathrm{W}).
$$
Hence, it follows that
$\widetilde{\S}(\alpha)\leq\S(\alpha)$.

Next, we  establish the opposite inequality, namely $\S(\alpha)\leq\widetilde{\S}(\alpha)$. For this purpose, we consider $\{v_k\}_{k \in \mathbb{N}}$  to be  a minimizing sequence for $\widetilde{\S}(\alpha)$, i.e.,
$$
  \widetilde{\S}(\alpha)=\lim_{k\to\infty}\Phi_{G,\Omega}(|\nabla v_k|)
+\Phi_{G,\Omega}(v_k) \qquad \text{and} \qquad \mathcal{H}^{N-1}(\{v_k=0\})\geq\alpha.
$$
Now, for any $k \in \mathbb{N}$ we choose $\mathrm{W}_k\subset\{v_k=0\}$ such that
$\mathcal{H}^{N-1}(\mathrm{W}_k)=\alpha$ (it is possible due to regularity of the Hausdorff measure, see \cite[Ch. 1]{LinYang}). Hence, we get
$$
\S(\alpha)\leq S_{G,H}(\mathrm{W}_k)\leq \Phi_{G,\Omega}(|\nabla v_k|)
+\Phi_{G,\Omega}(v_k).
$$
By taking limit when $k\to\infty$, we conclude that $\S(\alpha)\leq\widetilde{\S}(\alpha)$ as desired.
\end{proof}

Finally, we are in   position to supply for the proof of Theorem \ref{Thm1}.

\begin{proof}[{\bf Proof of Theorem \ref{Thm1}}]

Firstly, we will prove   statement \ref{Stat1}. Let $\{v_k\}_{k \in \mathbb{N}}$ be a nonnegative minimizing
sequence of $\S(\alpha)$, i.e,
$$
v_k \geq 0, \,\,\, \Phi_{H,\partial\Omega}(v_k)=1, \,\,\,  \mathcal{H}^{N-1}(\{v_k=0\})\geq\alpha
$$
and
$$
  \lim_{k \to \infty}\Phi_{G,\Omega}(|\nabla v_k|) +\Phi_{G,\Omega}(v_k)=\S(\alpha).
$$
So, $\{v_k\}_{k \in \mathbb{N}}$ is bounded in $W^{1,G}(\Omega)$ then by the reflexivity of the space $W^{1,G}(\Omega)$ there exists $u_0\in W^{1,G}(\Omega)$ such that
$$
v_k\rightharpoonup u_0 \quad \mbox{ weakly in }
W^{1,G}(\Omega).
$$
And from Theorem \eqref{compacidad} we know that
\begin{align*}
\label{weakly}
&v_k\to u_0\quad \mbox{ strongly in } L^{G}(\Omega),\\
& v_k\to u_0\quad \mbox{ strongly in } L^{H}(\partial\Omega),\\
& v_k\to u_0\quad \mathcal{H}^{N-1}\text{ a.e. on }\partial\Omega.
\end{align*}
From these limits we obtain that $\Phi_{H,\partial\Omega}(u_0)=1$ (see, again, Remark \eqref{equivmodulares}) and
\begin{equation} \label{eq.u0}
\mathcal{H}^{N-1}(\{u_0=0\})\geq\limsup_{n\to\infty}\mathcal{H}^{N-1}(\{v_k=0\})\geq\alpha.
\end{equation}
Then, $u_0$ is an admissible function for $\S(\alpha)$ (according to Lemma \ref{lemma.caract}), so
$$
\S(\alpha)\leq\Phi_{G,\Omega}(|\nabla u_0|) +\Phi_{G,\Omega}(u_0).
$$
 The another inequality  easily follows from the weak convergence in $W^{1,G}(\Omega)$ and the weakly lower semi-continuity of the functional $v \mapsto \Phi_{G,\Omega} (|\nabla v|)+\Phi_{G,\Omega}(v)$, i.e.,
$$
\Phi_{G,\Omega}(|\nabla u_0|)
+\Phi_{G,\Omega}( u_0 )\leq\lim_{k\to\infty}\Phi_{G,\Omega}(|\nabla v_k|)
+\Phi_{G,\Omega}(v_k)=\mathcal{S}(\alpha).
$$
Therefore, the function $u_0$ fulfills statement \ref{Stat1}.

Next, we will prove that statement \eqref{Stat1} implies statement \eqref{Stat2}. Indeed, from statement \eqref{Stat1} we know that there exists $u_0\in \mathrm{X}$ such that
$$
   \mathcal{H}^{N-1}(\{u_0=0\})\geq\alpha \quad \text{and} \quad \mathcal{S}(\alpha) = \Phi_{G,\Omega}( |\nabla u_0| )
+\Phi_{G,\Omega}(u_0).
$$
From this and the regularity of the corresponding Hausdorff measure (see \cite[Ch. 1]{LinYang}), there exists a closed set $\mathrm{W}_0\subseteq\{x\in\partial\Omega\colon u_0(x)=0\}$ such that
$$
   \mathcal{H}^{N-1}(\mathrm{W}_0)=\alpha.
$$
Consequently, by using \eqref{eqMin} we obtain that
$$
S_{G,H}(\mathrm{W}_0)\leq\Phi_{G,\Omega}(|\nabla u_0|) +\Phi_{G,\Omega}(u_0)=\S(\alpha) \leq S_{G,H}(\mathrm{W}_0).
$$
Therefore, $\S(\alpha)=S_{G,H}(\mathrm{W}_0)$, thereby finishing the proof.
\end{proof}

In the next result, namely Theorem \ref{ThmCharac}, we find that the set of zeros of the minimizers coincide exactly with the optimal window.

\begin{proof}[{\bf Proof of Theorem \ref{ThmCharac}}]
Let $u_0\in \mathrm{X}$ be an extremal for $\S(\alpha)$, then, from Lemma \ref{lemma.caract} it fulfills that
$$
  \mathcal{H}^{N-1}(\{u_0=0\})\geq \alpha \qquad\text{and}\qquad   \S(\alpha)=\Phi_{G,\Omega}(|\nabla u_0|) +\Phi_{G,\Omega}(u_0).
$$
Suppose for sake of contradiction that
$$
  \mathcal{H}^{N-1}(\{u_0=0\})> \alpha.
$$
Since the Hausdorff measure $\mathcal{H}^{s}$ is Borel regular for $(0\leq s<\infty)$ (cf. \cite[Ch. 1]{LinYang}) there exists a closed set $\mathrm{W}_0\subset \{x\in \partial\Omega \colon u_0(x)=0\}$ such that
$$
\S(\alpha)\leq S_{G,H}(\mathrm{W}_0).
$$
On the other hand, note that $u_0$ is an  admissible function in the characterization of $S_{G,H}(\mathrm{W}_0)$, from where
$$
S_{G,H}(\mathrm{W}_0)\leq \Phi_{G,\Omega}(|\nabla u_0|) +\Phi_{G,\Omega}(u_0).
$$
Consequently, $\S(\alpha)=S_{G,H}(\mathrm{W}_0)$, and so $u_0$ is also a minimizer of $S_{G,H}(\mathrm{W}_0)$. Hence, $u_0$ verifies in the weak sense
\begin{equation}\label{problema}
\left\{
\begin{array}{rclcl}
  -\div\left(\frac{G^{\prime}(|\nabla u_0|)}{|\nabla u_0|}\cdot \nabla u_0\right) + \frac{g( u_0 )u_0}{|u_0|} & = & 0 & \text{in} & \Omega \setminus \{u_0=0\}\\
  g(|\nabla u_0|) \frac{\nabla u_0}{|\nabla u_0|}\cdot \eta & = &  S_{G,H}(\mathrm{W}_0)
\frac{h(u_0)u_0}{|u_0|} & \text{on} & \partial \Omega\setminus\mathrm{W}_0 \\
  u_0 & = & 0 & \text{on} & \mathrm{W}_0.
\end{array}
\right.
\end{equation}
From the available regularity theory for $u_0$ (cf. \cite{Kor} and
\cite[Ch. 5]{Lieberman}) and Hopf type result (cf. \cite[Theorem 1]{MiSh} for subsolutions) we conclude
that
$$
\nabla u_0 \cdot \eta >0\text{ on } \partial (\{ x\in\partial\Omega\colon u_0(x)=0\}\setminus\mathrm{W}_0) \,\,\,(\text{point-wisely}),
$$
which contradicts the second equation in \eqref{problema}, thereby finishing the proof.
\end{proof}
As a consequence from Theorem \ref{ThmCharac} we obtain the
following monotonicity result:

\begin{corollary} The set mapping $\alpha \mapsto \S(\alpha)$ is a strictly increasing function.
\end{corollary}

\begin{proof}
From definition, it holds that $\alpha \mapsto \S(\alpha)$ is a nondecreasing mapping.
Next, suppose for sake of contradiction that there exists $0 < \alpha < \sigma < \mathcal{H}^{N-1}(\partial \Omega)$ such that $\S(\alpha) =\S(\sigma)$. In particular, this would imply that minimizers for $\S(\sigma)$ are also minimizers for $\S(\alpha)$. However, if $u_0$ is a minimizer for $\S(\sigma)$, then from Theorem \ref{ThmCharac} we get
$$
\mathcal{H}^{N-1}(\{u_0=0\}) = \sigma > \alpha,
$$
which clearly contradicts Theorem \ref{ThmCharac} when applied for $\S(\alpha)$. Therefore, $\S$ is a strictly increasing map.
\end{proof}

\section{Optimal interior holes: Proof of Theorems \ref{Thm4} and \ref{Thm5}}\label{OptProbHole}

In this section we prove existence and some further properties on optimal holes with a prescribed volume.

The next proof runs similarly to  \cite[Theorem 1.2]{BRW2}. We will include the details for the reader convenience.
\begin{proof}[{\bf Proof of Theorem \ref{Thm4}}]
Let $0<\alpha<\Leb(\Omega)$. It is easy to prove that $\mathbb{S}(\alpha)$ is obtained minimizing
$\mathbb{S}_{G,H}(\mathrm{A})$ over all the subsets
$\mathrm{A}\subset \Omega$ such that $\Leb(\mathrm{A})\geq
\alpha$. Moreover, it is clear that
$$
  \inf\left\{\mathbb{S}_{G,H}(\mathrm{A}): \mathrm{A} \subset \Omega, \Leb(\mathrm{A})=\alpha  \right\} \geq \inf\left\{\mathbb{S}_{G,H}(\mathrm{A}): \mathrm{A} \subset \Omega, \Leb(\mathrm{A})\geq \alpha  \right\}.
$$
Observe that test functions for a set of measure greater than or
equal to $\alpha$ are also test functions for a set of measure
$\alpha$, from where the two infimum above coincide.

Now, let us prove that
\begin{equation} \label{S1}
  \mathbb{S}(\alpha)= \inf\left \{
 \Phi_{G,\Omega}(|\nabla u|) +\Phi_{G,\Omega}(u) \colon u\in W^{1,G}(\Omega),
\Phi_{H,\partial\Omega}(u)=1\text{ and }
\Leb(\{u=0\})\geq\alpha\right\}.
\end{equation}
Since the minimizer does not change sign, test functions can be considered to be nonnegative. As in the proof of Theorem \ref{Thm1}, given a nonnegative minimizing sequence $\{v_k\}_{n\in\N}$ of $\mathbb{S}(\alpha)$ there exists $u_0\in W^{1,G}(\Omega)$ such that
\begin{align*}
&v_k\rightharpoonup u_0 \text{ weakly  in }   W^{1,G}(\Omega),\\
&v_k\rightharpoonup u_0 \text{ strongly  in }   L^G(\Omega),\\
&v_k\to u_0 \text{ strongly  in }L^{H}(\partial \Omega),\\
&v_k\to u_0\text{ a.e. in }\partial \Omega.
\end{align*}
Moreover, $ \Phi_{H,\partial\Omega}(u_0)=1$ and $u_0$ can be assumed
to be nonnegative.

Up to a subsequence, given the sets $\mathrm{A}_k=\{v_k=0\}$, there exists a function $0\leq \phi \leq 1$ such that  $\chi_{\mathrm{A}_k}\rightharpoonup\phi$ weakly in the dual space $ L^{\tilde{G}}(\Omega)$. Particularly, for $\mathrm{A}=\{\varphi>0\}$ it holds that
$$
\Leb(\mathrm{A})\geq \int_\Omega \phi(x) \,dx = \lim_{k\to \infty} \int_\Omega \chi_{\mathrm{A}_k}(x)\,dx =\Leb(\mathrm{A}_k)\geq \alpha.
$$
Since $u_0$ and $\phi$ are nonnegative, and
$$
   \int_\Omega u_0(x)\phi(x)\,dx=\lim_{k\to \infty} \int_\Omega v_k(x)\chi_{\mathrm{A}_k}(x)\,dx=0,
$$
there holds that  $u_0=0$ a.e. in $\mathrm{A}$. Then, $u_0$ is an admissible function for $\mathbb{S}(\alpha)$ and
$$
     \mathbb{S}(\alpha)\leq\Phi_{G,\Omega}(|\nabla u_0|) +\Phi_{G,\Omega}( u_0 ).
$$
The weak convergence in $W^{1,G}(\Omega)$ and the weakly lower semi-continuity of the functional $v \mapsto \Phi_{G,\Omega} (|\nabla v|)+\Phi_{G,\Omega}(v)$ gives that
$$
\Phi_{G,\Omega}(|\nabla u_0|)
+\Phi_{G,\Omega}( u_0 )\leq\lim_{k\to\infty}\Phi_{G,\Omega}(|\nabla v_k|)
+\Phi_{G,\Omega}(v_k)=\mathbb{S}(\alpha)
$$
and therefore \eqref{S1} holds.

Finally, it only remains to prove that $\Leb(\{u_0=0\})=\alpha$. Suppose for the  sake of contradiction that $u_0\equiv 0$ in a set $\mathrm{A}$ with $\Leb(\mathrm{A})>\alpha$. By taking a subset we may assume that $\mathrm{A}$ is closed. Now, let $B$ be a small ball such that $\Leb(\mathrm{A}\setminus B)>\alpha$ with $B$ centered in a point in $\partial \mathrm{A} \cap \partial \Omega_1$, where $\Omega_1$ is the connected component of $\Omega\setminus \mathrm{A}$ such that $\partial \Omega\subset \partial \Omega_1$. Notice that, we can pick the ball $B$ such that $\Leb(\mathrm{A}\cap B)>0$. Particularly $\Leb(\{u_0=0\}\cap B)>0$.

Now, since $u_0$ is an extremal for $\mathbb{S}(\alpha)$ and
$\Leb(\mathrm{A}\setminus B)>\alpha$, it is also an
extremal for $\mathbb{S}_{G,H}(A\setminus B)$. Thus,
$$
-\div\left(g(|\nabla u_0|)\frac{\nabla u_0}{|\nabla u_0|}\right) + \frac{g(u_0)u_0}{|u_0|} = 0 \quad \text{ in } \quad \Omega\setminus (\mathrm{A}\setminus B).
$$
As $u_0\geq 0$, in view of Remark \ref{remarkk} (see item (3)), either $u\equiv 0$ or $u_0>0$ in each connected component of $\Omega\setminus (\mathrm{A}\setminus B)=(\Omega \setminus \mathrm{A})\cup B$. Since $u_0\neq 0$ on $\partial \Omega$, in particular, $u_0>0$ in $B$, which clearly contradicts the choice of the ball $B$, and therefore $\Leb(\{u_0=0\})=\alpha$. This concludes the proof.
\end{proof}

Now, we will establish that our optimization problem prevents the
existence of maximal interior holes.

\begin{proof}[{\bf Proof of Theorem \ref{Thm5}}]
For $0\leq \varepsilon \ll 1$ fixed, let $\delta = \delta(\varepsilon)$ be such that
$$
  \mathrm{A}_{\varepsilon, \delta} = \{x \in \Omega: \varepsilon \leq \dist(x, \partial \Omega)\leq \delta\}
$$
fulfills $\Leb(\mathrm{A}_{\varepsilon, \delta}) = \alpha$. We affirm that
$$
  \mathbb{S}_{G, H}(\mathrm{A}_{\varepsilon, \delta}) \to +\infty \quad \text{as} \quad \varepsilon \to 0^+.
$$
Indeed, let $v_{\varepsilon} \in W^{1,G}(\Omega)\setminus W_0^{1,G}(\Omega)$ be a minimizer for $\mathbb{S}_{G, H}(\mathrm{A}_{\varepsilon, \delta})$  according to definition \eqref{EqMinHole} and remember that we   assumed the normalization   $\Phi_{H, \partial \Omega}(v_{\varepsilon}) = 1$.

Now, for each $\sigma>0$ fixed, we consider
$$
\Omega_{\sigma} = \{x \in \Omega: \dist(x, \partial \Omega)> \sigma\}.
$$
Notice that $v_{\varepsilon}$ is a weak solution to
$$
\left\{
\begin{array}{rclcl}
  -\div\left( g(|\nabla v_{\varepsilon}|)\cdot\frac{\nabla v_{\varepsilon}}{|\nabla v_{\varepsilon}|}\right) + \frac{g( v_{\varepsilon} )}{ |v_{\varepsilon}|}v_{\varepsilon}& = & 0 & \text{in} & \Omega_{\delta}\setminus \{v_{\varepsilon} = 0\}\\
  v_{\varepsilon} & = & 0 & \text{on} & \mathrm{A}_{\varepsilon, \delta}.
\end{array}
\right.
$$
In particular, $v_{\varepsilon} = 0$ on $\partial \Omega_{\delta}$. Hence, from Comparison Principle (cf. \cite{Mon}) $v_{\varepsilon} = 0$ in $\Omega_{\delta}$. Furthermore, remember that by construction $v_{\varepsilon} = 0$ in $\mathrm{A}_{\varepsilon, \delta}$, therefore
\begin{equation}\label{EqThm3}
v_{\varepsilon} \to 0 \quad \text{a.e.} \quad \Omega \quad \text{as} \quad \varepsilon \to 0.
\end{equation}
Finally, supposing for the sake of contradiction that $\mathbb{S}_{G, H}(\mathrm{A}_{\varepsilon, \delta})$ is  bounded. Then, up to a subsequence, there would exist a $u_0 \in W^{1,G}(\Omega)$ such that
\begin{eqnarray}\label{weakly}
  v_{\varepsilon}\rightharpoonup u_0 & \mbox{ weakly in } & W^{1,G}(\Omega),\\\label{weakly1}
  v_{\varepsilon}\to u_0 & \mbox{ strongly in } &  L^{G}(\Omega),\\\label{weakly2}
  v_{\varepsilon} \to u_0 &  \mbox{ strongly in } &  L^{H}(\partial\Omega),\\\label{weakly3}
  v_{\varepsilon}\to u_0 &  \mathcal{H}^{N-1}\text{ a.e. on } & \partial\Omega.\label{weakly4}
\end{eqnarray}
Taking into account the sentences \eqref{EqThm3}, \eqref{weakly4}, and the normalization condition $\Phi_{H, \partial \Omega}(v_{\varepsilon}) = 1$ we obtain a contradiction.
\end{proof}

\section{Proof of Theorems \ref{Thm1.6} and \ref{Thm1.7}}\label{Sec6}

This section will deal with general shapes that may have zero Lebesgue measure. First, we  analyze when the ``Sobolev trace constant'' perceives the set $\mathrm{A}\subset \R^n$ with zero Lebesgue measure. In this direction, we prove a continuity result of $\mathbb{S}_{\mathrm{A}}$ in relation to $\mathrm{A}$ in Hausdorff distance.

In order to study when $\mathbb{S}_{\mathrm{A}}$ recovers the usual Orlicz-Sobolev trace
constant, i.e., when $\mathbb{S}_{\mathrm{A}}=\mathbb{S}_{\emptyset}$, we recall the
notion of $G-$capacity for Young functions, which plays a fundamental role in Nonlinear Potential Theory.

\begin{definition}[{\bf $G-$Capacity, \cite[Definition 2.2]{Bieg}}]
Given a Young function $G$ satisfying the $\Delta_2-$condition, we define the $G-$capacity of $\mathrm{A}\subset \R^N$ by
{\small{
$$
   \text{Cap}_G(\mathrm{A})=\inf \left\{\int_{\R^N} G( \varphi )+G(|\nabla \varphi|)\,dx : \varphi\in W^{1,G}(\R^N)\cap C^\infty(\R^N) \quad \text{and} \quad \varphi|_{\mathrm{A}}\geq 1\right\}.
$$}}
Moreover, if such a function does not exists, then $\text{Cap}_G(\mathrm{A})=\infty$.
\end{definition}

Next, we will deliver the proof of Theorem \ref{Thm1.6}.

\begin{proof}[{\bf Proof of Theorem \ref{Thm1.6}}]
We prove both implications of the statement.
\begin{enumerate}
  \item Let us see that $\mathbb{S}_{\mathrm{A}} = \mathbb{S}_\emptyset$ implies that  $\text{Cap}_G(\mathrm{A})=0$.

Given an extremal $u_0$ for $\mathbb{S}_{\mathrm{A}}$, it is also an
extremal for $\mathbb{S}_\emptyset$, and therefore it is a weak
solution to
$$
  \left\{
\begin{array}{rclcl}
  -\div\left(\frac{g(|\nabla u_0|)}{|\nabla u_0|}.\nabla u_0\right) + \frac{g( u_0 )u_0}{|u_0|} & = & 0 & \text{in} & \Omega \\
  g(|\nabla u_0|)\frac{\nabla u_0 }{|\nabla u_0|}\cdot\eta & = &  \mathbb{S}_{\mathrm{A}}
\frac{h( u_0  )u_0 }{|u_0 |} & \text{on} & \partial \Omega.
\end{array}
\right.
$$
From Remark \ref{remarkk} items (2, 4), we have that $u_0\in C^{1,\gamma}(\bar \Omega)$ and $u_0>0$ in $\bar\Omega$.

Since $u_0\in W_{\mathrm{A}}^{1,G}(\Omega) = {\overline{C^\infty_0(\bar\Omega\setminus \mathrm{A})}}^{\|.\|_{W^{1,G}}}$, there exists a sequence $\{u_k\}_{k\in\N} \subset C^\infty_0(\bar\Omega\setminus A)$ such that $u_k\to u_0$ in $W^{1,G}(\Omega)$.

Now, observe that the function $\varphi_k=1-\frac{u_k}{u_0}$ is identically equal to $1$ in a neighborhood of $\mathrm{A}$. Moreover, since $\beta=\inf\{u_0(x): x\in \bar \Omega\}>0$, we get
$$
  \|\varphi_k\|_{L^G(\Omega)}=\left\|\frac{u_0-u_k}{u_0}\right\|_{L^G(\Omega)} \leq \frac{1}{\beta} \|u_0-u_k\|_{L^G(\Omega)}
$$
and
\begin{align*}
\|\nabla \varphi_k\|_{L^G(\Omega)}&=\left\| \frac{1}{u_0}\nabla u_k - \frac{u_k}{u_0^2}\nabla u_0 \right\|_{L^G(\Omega)}\\
&=\left\| \frac{1}{u_0}\nabla (u_k-u_0) + \frac{1}{u_0^2}\nabla u_0 (u_0-u_k)\right\|_{L^G(\Omega)}\\
&\leq \frac{1}{\beta} \|\nabla (u_k-u_0)\|_{L^G(\Omega)}+ \frac{1}{\beta^2}\|\nabla u_0\|_{L^\infty(\Omega)}\|u_0-u_k\|_{L^G(\Omega)}.
\end{align*}
Therefore, for every $\ve>0$ there exists $k_0\geq 1$ such that
$$
   \| \varphi_k\|_{W^{1,G}(\Omega)}<\ve
$$
if $k\geq k_0$. Finally, the proof finishes by using \cite[Theorem 8.5.7]{FucJohnKuf} after extending $\varphi_k$ to $\R^N$ and regularizing.

  \item Let us see that $\text{Cap}_G(\mathrm{A})=0$ implies $\mathbb{S}_{\mathrm{A}} = \mathbb{S}_{\emptyset}$.

Let us prove that if $\text{Cap}_G(\mathrm{A})=0$ then
$$
  W^{1,G}_A(\Omega)=W^{1,G}_\emptyset(\Omega)=W^{1,G}(\Omega),
$$
from where it will follows the lemma. Let $\mathrm{A} \subset
\Omega$ be such that $\text{Cap}_G(\mathrm{A})=0$. Then, given
$\ve>0$ there exists $\varphi_\ve \in W^{1,G}(\R^N)\cap
C^\infty(\R^N)$ such that
$$
   \int_{\R^N} G(|\varphi_\ve|)+G(|\nabla \varphi_\ve|)\,dx<\ve
$$
and
$\varphi_\ve\equiv 1$ in a neighborhood of $\mathrm{A}$.

Take $u\in C^\infty(\bar\Omega)$ and define $u_\ve=(1-\varphi_\ve)u$. Observe that $u_\ve \in W^{1,G}(\Omega)$ and $u_\varepsilon =0$ in $\mathrm{A}$. Hence,
$$
  \|u-u_\ve\|_{W^{1,G}(\Omega)}=\|\varphi_\ve u\|_{W^{1,G}(\Omega)}.
$$
Since $W_{\mathrm{A}}^{1,G}(\Omega) = {\overline{C^\infty_0(\bar\Omega\setminus \mathrm{A})}}^{\|.\|_{W^{1,G}}}$ the result follows if we show that $\|\varphi_\ve u\|_{W^{1,G}(\Omega)}$ vanishes as $\ve\to0$.

Now, since the $\Delta_2$ condition is in force, from  \cite[3.10.4]{FucJohnKuf} the convergence in norm is equivalent to the convergence of modulars. Consequently, we have to show that
$$
   \int_{\R^N} G(|\nabla(\varphi_\ve u)|) + G(|\varphi_\ve u|)\,dx \to 0 \quad \text{as } \ve\to 0.
$$
For that end, we have to bound each term of the expression
$$
   \int_{\R^N} G(|\nabla(\varphi_\ve u)|) + G( |\varphi_\ve u |)\,dx\leq C \int_{\R^N} G(|u\nabla \varphi_\ve|)+ G(|\varphi_\ve \nabla u|) + G( |\varphi_\ve u| )\,dx.
$$
where we have used the $\Delta_2$ condition. First, since $G$ is
increasing and the  Sobolev extension theorem holds, using again the
$\Delta_2$ condition,
{\small{
$$
   \int_{\R^N} G(|u\nabla \varphi_\ve|)\,dx \leq  \int_{\R^N} G(|\|u\|_\infty \nabla \varphi_\ve|)\,dx \leq C_\gamma\|u\|_\infty^\gamma \int_{\R^N} G(|\nabla \varphi_\ve|)\,dx \leq C_\gamma \|u\|_\infty^\gamma \ve
$$}}
where $\gamma$ is a constant depending only  on $g^\pm$.

The second and third terms are bounded in the same way, and the proof concludes.
\end{enumerate}
\end{proof}

Before proving Theorem \ref{Thm1.7}, let us recall the definition of Hausdorff distance.

\begin{definition} Let $X, Y \subset \R^N$ be two non-empty subsets. The Hausdorff distance between $X$ and $Y$ is given by
$$
 \displaystyle \text{dist}_{\mathcal{H}}(X, Y):= \inf\left\{\varepsilon>0:\,\,  X \subset Y_{\varepsilon} \quad \text{and} \quad Y \subset X_{\varepsilon}\right\},
$$
where $\displaystyle Z_{\varepsilon} = \bigcup_{z_0 \in Z} \left\{z \in \R^N /  \dist(z, z_0) \leq \varepsilon\right\}$ is the usual fattening of $Z$.
\end{definition}

We are now in position to supply for the proof of Theorem \ref{Thm1.7}

\begin{proof}[{\bf Proof of Theorem \ref{Thm1.7}}]
Let $\displaystyle \mathrm{A}_\varepsilon=\bigcup_{x\in\mathrm{A}} B_\varepsilon(x)$. Assume that
$\dist_{\mathcal{H}}(\mathrm{\mathrm{A}_k}, \mathrm{A})\to 0$ as $k\to\infty$, then given $\varepsilon>
0$ there exists $k_0 \in \mathbb{N}$ such that $\mathrm{A},\mathrm{A_k}\subseteq
A_\varepsilon$ if  $k\geq  k_0$ and it follows that
$$
W^{1,G}_
{\mathrm{A}_\varepsilon}(\Omega)\subseteq W^{1,G}_{\mathrm{A}}(\Omega)\cap W^{1,G}_
{\mathrm{A_k}}(\Omega).
$$
First, note that
$$
\mathbb{S}_{\emptyset}\leq\mathbb{S}_{\mathrm{A}},\, \mathbb{S}_{\mathrm{A_k}}\leq
\mathbb{S}_{\mathrm{A_\varepsilon}} \quad \mbox{ if } k\geq k_0.
$$
Now, let $u\in W^{1,G}_{\mathrm{A}}(\Omega)$ be an extremal for
$\mathbb{S}_{\mathrm{A}}$ normalized such that $\Phi_{H, \partial
\Omega}(u) = 1$. As  $u\in W_{\mathrm{A}}^{1,G}(\Omega) = {\overline{C^\infty_0(\bar\Omega\setminus \mathrm{A})}}^{\|.\|_{W^{1,G}}}$, given $\delta> 0$ there exists $u_\delta\in
C^\infty_0(\overline{\Omega}\setminus A)$ such that
$$
  \|u-u_\delta\|_{W^{1,G}(\Omega)}\leq\delta.
$$
Moreover, we may suppose that
$$
  supp(u_\delta)\subseteq\overline{\Omega}\setminus \mathrm{A}_\varepsilon
$$
if $\varepsilon$ is small enough. Now $u_\delta\in
W^{1,G}_{\mathrm{A_k}}(\Omega)$ for $k\geq k_0$, so using Lemma \ref{Lemma2.3}
\begin{align*}
\mathbb{S}_{\mathrm{A_k}}&\leq
\Phi_{G,\Omega}(u_\delta)+\Phi_{G,\Omega}(|\nabla u_\delta|)\\
&\leq (1+\eta)^{g^+}\left[\Phi_{G,\Omega}(u)+ \Phi_{G,\Omega}(|\nabla u|)\right]+ C_\eta\left[\Phi_{G,\Omega}(u-u_\delta)+\Phi_{G,\Omega}( |\nabla (u-u_\delta)|)\right]\\
&\leq (1+\eta)^{g^+}\mathbb{S}_{\mathrm{A}}+C_\eta\|u-u_\delta\|_{W^{1,G}(\Omega)}^{g^-}\\
&\leq (1+\eta)^{g^+}\mathbb{S}_{\mathrm{A}}+C_\eta\delta^{g^-}
\end{align*}
where $\eta\geq 0$ is an arbitrary parameter and we have used that
$$
\Phi_{G,\Omega}\left(\frac{u-u_\delta}{\|u-u_\delta\|_{G}}\right) \leq 1 \implies \Phi_{G,\Omega}\left(u-u_\delta\right) \leq \|u-u_\delta\|_{L^G(\Omega)}^{g^-} \leq \|u-u_\delta\|_{W^{1, G}(\Omega)}^{g^-}.
$$
Then,  taking first limit as $\delta\to0$, we obtain
$$
\mathbb{S}_{\mathrm{A_k}}\leq (1+\eta)^{g^+}\mathbb{S}_{\mathrm{A}},
$$
and then,  $\eta\to0$, we get
$$
\mathbb{S}_{\mathrm{A_k}}\leq \mathbb{S}_{\mathrm{A}}.
$$
Analogously, it can be proved that the reverse inequality
$
  \mathbb{S}_{\mathrm{A}}\leq \mathbb{S}_{\mathrm{A_k}},
$
and then
$$
 |\mathbb{S}_{\mathrm{A_k}}-\mathbb{S}_{\mathrm{A}}|\to 0 \quad \text{as } k \to \infty.
$$
Let $u_k$ be an extremal for $\mathbb{S}_{\mathrm{A_k}}$ normalized
such that $\Phi_{G,H}(u_k)=1$. Then,
$$
\Phi_{G,\Omega}(u_k)+\Phi_{G,\Omega}(|\nabla
u_k|)=\mathbb{S}_{\mathrm{A_k}}\to\mathbb{S}_{\mathrm{A}} \quad \text{as }k\to\infty.
$$
It follows that $u_k\subseteq W^{1,G}(\Omega)$ is bounded. Hence,  there
exists a subsequence (still denoted by $u_k$) and a function
$u\in W^{1,G}(\Omega)$ such that
\begin{align*}
&u_k\rightharpoonup u \mbox{ weakly in } W^{1,G}(\Omega),\\
&u_k\rightharpoonup u \mbox{ strongly in } L^{H}(\partial\Omega).
\end{align*}
By definition of the spaces $W^{1,G}_{\mathrm{A_k}}(\Omega)$ and
$W^{1,G}_{\mathrm{A}}(\Omega)$  it is easy to see that $u\in
W^{1,G}_{\mathrm{A}}(\Omega)$. Moreover, $\Phi_{G,H}(u)=1$ and  also
$$
\begin{array}{rcl}
  \mathbb{S}_{\mathrm{A}}\leq
\Phi_{G,\Omega}(u)+\Phi_{G,\Omega}(|\nabla u|)& \leq & \displaystyle \lim_{k \to \infty}
\Phi_{G,\Omega}(u_k)+\Phi_{G,\Omega}(|\nabla
u_k|)\\
 & = & \displaystyle \lim_{k \to \infty} \Phi_{G,\Omega}(u_k)+\Phi_{G,\Omega}(|\nabla u_k|) \\
   & \leq & \displaystyle \lim_{k \to \infty}\mathbb{S}_{\mathrm{A_k}}=\mathbb{S}_{\mathrm{A}}.
\end{array}
$$
The proof is concluded.
\end{proof}

\begin{example} Observe that when $H=G$ in our results, we can recover, to some extent the ``Steklov eigenvalue problem''. However, in contrast with the power case, for general Young functions it is not immediate that the hypothesis in Theorem \ref{compacidad} is fulfilled.

Let us see that, in fact, that such condition is satisfied, i.e, $G<<(G^*)^\frac{N-1}{N}$. Indeed, it is equivalent to $G^\frac{N}{N-1}<< G^*$. Moreover, since $G$ is increasing it also is equals to prove the following limit
\begin{equation}\label{equiv}
\lim_{t\to+\infty} \frac{(G^*)^{-1}(t)}{(G^\frac{N}{N-1})^{-1}(t)}=0
\end{equation}
to hold (see \cite[p.291]{AdamsFour}).

First observe that integrating by part we obtain that
$$
 (G^*)^{-1}(t)= \int_0^{t} \frac{G^{-1}(s)}{s^{1+\frac{1}{N}}}ds
 =N\int_0^{t} \frac{(G^{-1}(s))'}{s^{\frac{1}{N}}}ds -\frac{N }{t^\frac{1}{N}} G^{-1}(t).
$$
We divide the previous expression by $(G^\frac{N}{N-1})^{-1}(t)$, and analyze each one of the its two terms. Notice that $(G^\frac{N}{N-1})^{-1}(t) = G^{-1}(t^\frac{N-1}{N})$.

From the L'Hospital's rule, the first term behaves as follows
 \begin{align*}
 I_1:=\lim_{t\to\infty} \frac{N}{G^{-1}(t^\frac{N-1}{N})} \int_0^{t} \frac{(G^{-1}(s))'}{s^{\frac{1}{N}}}ds &=
 \lim_{t\to\infty} \frac{(G^{-1}(t))'}{t^{\frac{1}{N}}} \frac{N}{(G^{-1}(t^\frac{N-1}{N}))'}\\
 &=\lim_{t\to\infty} \frac{N}{t^\frac1N} \frac{G'(G^{-1}(t^\frac{N-1}{N}))}{G'(G^{-1}(t))}.
 \end{align*}
It is not hard to prove that $G^{-1}$ satisfies a condition analogous to \eqref{Liebcond} with constants $(g^+)^{-1}$ and $(g^-)^{-1}$, see \cite[Lemma 2.2]{MW}. Then
$$
 \frac{N}{t^\frac1N}\frac{G'(G^{-1}(t^\frac{N-1}{N}))}{G^{\prime}(G^{-1}(t))}  \leq
 \frac{g^+}{g^-}  \frac{N}{t^\frac1N} \frac{G(G^{-1}(t^\frac{N-1}{N}))}{G(G^{-1}(t))} = \frac{g^+}{g^-} \frac{N}{t^\frac{2}{N}} \frac{G^{-1}(t)}{G^{-1}(t^{1-\frac1N})}
$$
but again, using the $\Delta_2$ condition for $G^{-1}$, see \cite[Lemma 2.2]{MW}, the above inequality can be bounded as follows
$$
\frac{g^+}{g^-} \frac{N}{t^\frac{2}{N}} \frac{G^{-1}(t)}{G^{-1}(t^{1-\frac1N})} = \frac{g^+}{g^-} \frac{N}{t^\frac{2}{N}} \frac{G^{-1}(t^{1-\frac1N} t^\frac1N)}{G^{-1}(t^{1-\frac1N})} \leq \frac{g^+}{g^-} \frac{N t^\frac{1}{Ng^- } }{t^\frac{2}{N}} =  \frac{g^+}{g^-} N t^{\frac{1}{N}(\frac{1}{g^-}-1) }
$$
for $t>1$. Consequently, since $g^->1$, the last three expressions lead to $I_1=0$.

Let us deal with the second  term. Using again the $\Delta_2$ condition for $G^{-1}$ we get
$$
I_2:=\lim_{t\to\infty} \frac{N}{t^\frac1N} \frac{G^{-1}(t)}{G^{-1}(t^{1-\frac{1}{N}})}    =\lim_{t\to\infty} \frac{N}{t^\frac1N} \frac{G^{-1}(t^{1-\frac1N}t^\frac1N)}{G^{-1}(t^{1-\frac1N})} \leq  \lim_{t\to\infty} N   t^{\frac{1}{N}(\frac{1}{g^-}-1)}
$$
and it vanishes since $g^->1$. Since $I_1+I_2=0$, \eqref{equiv} holds as required.
\end{example}

\section{Spherical symmetrization}\label{SpheSym}

In this short section we will characterize the optimal window $\mathrm{W}_0 \subset \partial \Omega$ in our optimization problem as a spherical cap provided that $\Omega = B_1$. For that purpose, an essential tool is played by the \textit{spherical symmetrization}.

Given a mensurable set $\mathcal{E}\subset\R^N$, the spherical symmetrization $\mathcal{E}^{\sharp}$ of $\mathcal{E}$ with respect to an axis given by a unit vector $e_k$ reads as follows: for each positive number $r$, take the intersection $\mathcal{E}\cap \partial B(0,r)$ and replace it by the spherical cap of the same $\mathcal{H}^{N-1}-$measure and center $re_k$. Hence, $\mathcal{E}^{\sharp}$ is the union of these caps.

Now, the spherical symmetrization $u^{\sharp}$ of a measurable function $u: \Omega \to \R_{+}$ is constructed by symmetrizing the super-level sets so that, for all $t$
$$
   \{u^{\sharp}\geq t\} = \{u\geq t\}^{\sharp}.
$$
We recommend to the reader references \cite{Kawohl} and \cite{Sperner} for more details.

In the following, we recall some useful tools from Measure Theory.

\begin{definition} Let $(\mathbb{X}, \mathfrak{M}, \mu)$ be a measure space. Given a measurable function $f: \mathbb{X} \to \R$, the \textit{distribution function} of $f$ is the function $\varrho_f: [0, +\infty) \to [0,  \mu(\mathbb{X})]$ defined as follows:
$$
   \varrho_f(t) := \mu(\{x \in \mathbb{X}: |f(x)| > t\}).
$$
\end{definition}

The next result holds as consequence from Fubini-Tonelli's theorem.

\begin{lemma}\label{AuxLemma}   Let $(\mathbb{X}, \mathfrak{M}, \mu)$ be a finite measure space, let $G \in C^1(\R)$ a convex function and let $f: \mathbb{X} \to \R$ be a measurable function. Then,
$$
   \displaystyle \int_{\mathbb{X}} G(|f(x)|)d\mu(x) = \int_{0}^{+\infty} G^{\prime}(t)\varrho_f(t)dt.
$$
\end{lemma}

\begin{remark}\label{RemDistFunc} Another important piece of information for our approach is the following result from \cite[Proposition 5, item b]{Sperner}: if $u \in L^1(\mathbb{X})$ then on the Borel sets
$$
 \varrho_{u}(t) = \varrho_{u^{\sharp}}(t) \quad \forall \,\,\, t\geq 0.
$$
\end{remark}

Next, for the characterization of the optimal window we need the following result.

\begin{proposition}\label{reluu*} Let $u \in W^{1, G}(B_1)\cap W^{1, H}(\partial B_1)$ and $u^{\sharp}$ be its spherical symmetrization. Assume that
$$
  P_{B_1}(\{x \in B_1: u(x)>t\})\footnote{Here $P_{\Omega}$ means the perimeter, in the sense of De Giorgi, relative to $\Omega$. For such a concept in the general case, see \cite{LinYang}.} \geq \gamma \varrho_{u}^{\frac{N-1}{N}}(t),
$$
for some positive constant $\gamma \geq N \sqrt[N]{\omega_N}$ and any $t  \geq 0$. Then, $u^{\sharp} \in W^{1, G}(B_1)\cap W^{1, H}(\partial B_1)$. Moreover,
\begin{enumerate}
  \item\label{Stet1} $\displaystyle \int_{B_1}G(|u^{\sharp}|)\,dx=\int_{B_1}G(|u|)\, dx,$
  \item\label{Stet2} $\displaystyle \int_{\partial B_1} H(|u^{\sharp}|)\, d\mathcal{H}^{N-1}=\int_{\partial B_1}H(|u|) d\mathcal{H}^{N-1},$
  \item\label{Stet3} $\displaystyle \int_{B_1} G(|\nabla u^{\sharp}|)\, dx \leq \int_{B_1}G(|\nabla u|)\, dx$.
\end{enumerate}
\end{proposition}

\begin{proof} The statements \eqref{Stet1} and \eqref{Stet2} hold by combining Lemma \ref{AuxLemma} and Remark \ref{RemDistFunc}.

For the last statement, from \cite[Section 3]{Bramanti} we know that
\begin{equation}\label{EqSymG-Orl}
   \displaystyle \int_{B_1} G\left(\frac{|\nabla u^{\sharp}|}{\lambda_0}\right)\, dx \leq \int_{B_1}G(|\nabla u|)\, dx,
\end{equation}
for a constant $\lambda_0 = \frac{N \sqrt[N]{\omega_N}}{\gamma}$. From assumption on $\gamma$ we obtain that $0<\lambda_0 \leq 1$. In this case, from the convexity of $G$ and \eqref{EqSymG-Orl} we obtain that
$$
 \displaystyle \int_{B_1} G(|\nabla u^{\sharp}|)\, dx \leq \lambda_0\int_{B_1} G\left(\frac{|\nabla u^{\sharp}|}{\lambda_0}\right)\, dx \leq \int_{B_1} G(|\nabla u|)\, dx,
$$
which concludes the proof.
\end{proof}

Finally, we will present the proof of our symmetrization result.

\begin{proof}[{\bf Proof of Theorem \ref{SymmetThm}}]
Firstly, for a fixed $\alpha\in (0,\mathcal{H}^{N-1}(\partial B_1))$ Theorem \eqref{ThmCharac} assures that there exists a profile $u_0\in \mathrm{X}_{\mathrm{W}}$ such that
$$
  \mathcal{H}^{N-1}(\{u_0=0\})=\alpha \qquad\text{and}\qquad  \S(\alpha)=\Phi_{G, B_1}(|\nabla u_0|)+\Phi_{G, B_1}(u_0).
$$
Now, let $u_0^{\sharp}$ be the spherical symmetrization of $u_0$. Notice that $u_0^{\sharp}$ is an admissible profile in the optimization process of $\S(\alpha)$, and by the Proposition \eqref{reluu*}
$$
\S(\alpha)\leq \Phi_{G, B_1}(|\nabla u_0^{\sharp}|)+\Phi_{G, B_1}( u_0^{\sharp} )\leq \Phi_{G, B_1}(|\nabla u_0|)+\Phi_{G,B_1}( u_0 )=\S(\alpha).
$$
Therefore,
\begin{equation}\label{sym}
  \S(\alpha)=\Phi_{G, B_1}(|\nabla u_0^{\sharp}|)+\Phi_{G, B_1}( u_0^{\sharp} ).
\end{equation}
Finally, since
$$
   \mathcal{H}^{N-1}(\mathrm{W}_0)  = \mathcal{H}^{N-1}(\{u_0=0\})= \alpha,
$$
where $ \mathrm{W}_0=\{x\in \partial B_1: u_0^{\sharp}(x)=0\}$,
we conclude by using \eqref{sym} that
$$
\S(\alpha)=S_{G,H}(\mathrm{W}_0),
$$
which assures that $\mathrm{W}_0$ is an optimal window. As a direct consequence we obtain the desired symmetry result, because the optimal window $\mathrm{W}_0$ is a spherical cap.
\end{proof}

\begin{remark}
At this point, the following question arises: if $\Omega$ is symmetric, does $\mathrm{A}_0$ (in the case of optimal interior hole) inherit the symmetry of the domain? The answer is positive  in some scenarios, for instance, as proved previously, if $\Omega = B_1$,  then  $\mathrm{A}_0$ is spherically symmetric (cf. \cite{BRW2}, \cite{daSDelPR}, \cite{DelPBN} and \cite{Den99} for similar results). However, for general configurations of the domain, $\mathrm{A}_0$ is not necessarily radially symmetric (cf. \cite{BLDR}).
\end{remark}

\subsection*{Acknowledgements}

This paper was supported by grants  PROICO 031906, UNSL, CONICET PIP
11220150100032CO and UBACyT 20020130100283BA. The authors would like
to thank Prof. Juli\'{a}n Fern\'{a}ndez Bonder for reading the draft
this manuscript and providing insightful comments. J.V. da Silva
would like to thank the Department of Mathematics and FCNyE from
Universidad de Buenos Aires for providing an excellent working
environment and scientific atmosphere during his Postdoctoral
program. J.V. da Silva thanks also to IMASL (CONICET) and Department
of Mathematics from Universidad Nacional de San Luis for their warm
hospitality and for fostering a pleasant scientific atmosphere
during his visit where part of this manuscript was written. J.V. da
Silva, A.M. Salort and A. Silva are members of CONICET.



\end{document}